\documentclass{amsart}
\usepackage{amsmath, amssymb, dsfont, amsaddr}
\usepackage{xcolor}

\newtheorem{Proposition}{Proposition}
\newtheorem{Theorem}{Theorem}

\newtheorem{Lemma}{Lemma}
\newtheorem{Remark}{Remark}
\theoremstyle{definition}
\newtheorem{Definition}{Definition}
\newtheorem{Question}{Question}

\title[Coarse embeddings of quotients by finite group actions]{Coarse embeddings of quotients by finite group actions}
\author{Thomas Weighill}
\address{Department of Mathematics and Statistics, University of North Carolina at Greensboro}

\begin{document}

\maketitle

\begin{abstract}
    We prove that for a metric space $X$ and a finite group $G$ acting on $X$ by isometries, if $X$ coarsely embeds into a Hilbert space, then so does the quotient $X/G$. A crucial step towards our main result is to show that for any integer $k > 0$ the space of unordered $k$-tuples of points in Hilbert space, with the $1$-Wasserstein distance, itself coarsely embeds into Hilbert space. Our proof relies on establishing bounds on the sliced Wasserstein distance between empirical measures in $\mathbb{R}^n$.
\end{abstract}

\section{Introduction}

Coarse embeddings of metric spaces into Hilbert space are important in a range of theoretical and applied areas of mathematics. By a result of Yu~\cite{yu2000coarse}, if a discrete bounded geometry metric space $X$ admits a coarse embedding into a Hilbert space, then the coarse Baum-Connes Conjecture holds for $X$. A consequence of this is that if a finitely generated group $G$ with the word-metric coarsely embeds into Hilbert space, then the Novikov Conjecture holds for $G$. Properties which imply the existence of coarse embeddings into Hilbert space include finite asymptotic dimension, introduced by Gromov~\cite{gromov20041asymptotic}, and Yu's Property A~\cite{yu2000coarse}.

On the more applied side, many classical machine learning algorithms operate on the assumption that the input data are vectors in a (real) Hilbert space. To analyze more complicated data types, it is therefore necessary to map the data into a Hilbert space. A coarse embedding is among the weakest types of map one can consider which nonetheless provides some control on distances. A good example of a data type which is hard to embed into Hilbert space is given by persistence diagrams, which arise in topological data analysis. Recent results have shown that for most distance metrics, a coarse embedding of all persistence diagrams into Hilbert space is impossible~\cite{bubenik2020embeddings, wagner2021nonembeddability}. On the other hand, the space of diagrams with $n$ or less points for fixed $n$ was shown to have finite asymptotic dimension in \cite{mitra2021space}, and therefore coarse embeddings from this space into Hilbert space do exist. Invariant machine learning is a fast growing field which includes the problem of embedding orbit spaces $X/G$ into Euclidean space (see e.g.~\cite{cahill2024group} and the references there). For example, $X$ may be a set of shapes, and $G$ a set of rigid transformations which should not affect the output of a well-designed classifier.

In this paper we prove the following main result.

\begin{Theorem}\label{thm:main2}
    Let $X$ be a metric space which coarsely embeds into Hilbert space and let $X/G$ be the quotient of $X$ by the action of a finite group $G$ acting by isometries. Then $X/G$ coarsely embeds into a Hilbert space.
\end{Theorem}

Theorem \ref{thm:main2} is a natural addition to existing results for other coarse geometric properties. In particular, the following properties (all of which imply the existence of a coarse embedding into Hilbert space) are known to be preserved under quotients by finite group actions:
\begin{itemize}
    \item asymptotic dimension at most $n$ \cite{kasprowski2017asymptotic} (for proper metric spaces)
    \item finite asymptotic dimension \cite{dydak2016preserving}
    \item Asymptotic Property C \cite{dydak2016preserving} (which was defined in \cite{dranishnikov2000asymptotic})
    \item Yu's Property A~\cite{dydak2016preserving} (for bounded geometry metric spaces)
\end{itemize}

Note that by Theorem \ref{thm:main2}, if $X$ embeds into Hilbert space, so does $X^k/S_k$ for any $k \in \mathbb{N}$, where the symmetric group $S_k$ acts on $X^k$ by permuting coordinates, and the metric on $X^k$ is an $\ell^p$ metric. The quotient space $X^k/S_k$, sometimes called the symmetric product, can be viewed as the space of unordered $k$-tuples of points in $X$. Thus we obtain:

\begin{Theorem}\label{thm:main1}
    Let $X$ be a metric space which coarsely embeds into a Hilbert space and let $\mathcal{E}^k X$ be the space of unordered $k$-tuples of points in $X$. Then $\mathcal{E}^k X$ coarsely embeds into a Hilbert space.
\end{Theorem}

In fact, we will prove Theorem \ref{thm:main1} first and obtain Theorem \ref{thm:main2} as a consequence. The proof of Theorem \ref{thm:main1} uses an already existing concept from the area of optimal transport, namely the sliced Wasserstein distance introduced by Rabin et al.~\cite{rabin2012wasserstein}. A version of the sliced Wasserstein distance was already used in \cite{carriere2017sliced} to construct a kernel on persistence diagrams which, on the space of diagrams with at most $n$ points, induces a coarse embedding into Hilbert space. Our result relies on a higher-dimensional version of this idea.

We now outline the paper. In Section \ref{sec:prelim} below, we introduce the necessary definitions and concepts. In Section \ref{sec:bounds} we establish the main upper and lower bounds on sliced Wasserstein distance which will be needed. Section \ref{sec:emp} concerns the proof of Theorem \ref{thm:main1} and thereafter of Theorem \ref{thm:main2}. Some concluding remarks about connections to other embeddings are made in the final section.

\section{Preliminaries}\label{sec:prelim}

For a set $X$, we denote elements of the product $X^k$ in bold face, and denote the $i^{th}$ coordinate of an ordered $k$-tuple $\mathbf{v} \in X^k$ by $\mathbf{v}_i$.  Define $\langle k \rangle :=  \{1,\ldots,k\}$. We recall some basic definitions and results from optimal transport (see e.g.~\cite{peyre2019computational}). 

\begin{Definition}
Let $(X,d)$ be a metric space on which every finite Borel measure is a Radon measure. Given two Borel probability measures $\alpha$ and $\beta$ with finite $p^{\mathrm{th}}$ moments, the \emph{$p$-Wasserstein distance} is defined by
\[
W_p(\alpha, \beta) = \left( \inf_{\gamma \in \mathcal{U}(\alpha,\beta)} \int d(x,y)^p d\gamma(x,y) \right)^{1/p},
\]
where $\mathcal{U}(\alpha,\beta)$ is the set of Borel measures on $X^2$ with marginals $\alpha$ and $\beta$. 
\end{Definition}

If $\alpha = \sum_{i=1}^m \mathbf{a}_i \delta_{x_i}$ and $\beta = \sum_{i=1}^n \mathbf{b}_i \delta_{y_i}$ are discrete measures, then the Wasserstein distance can be equivalently formulated as
\[
W_p(\alpha, \beta) = \left( \min_{\mathbf{P} \in U(\mathbf{a}, \mathbf{b})} \sum_{i,j} d(x_i, y_j)^p \mathbf{P}_{ij} \right)^{1/p},
\]
where $\mathbf{P}$ ranges over the set $U(\mathbf{a},\mathbf{b})$ of $m\times n$ matrices such that $\mathbf{P} \mathds{1} = \mathbf{a}$ and $\mathbf{P}^T \mathds{1} = \mathbf{b}$.  We denote by $W_p(X)$ the \emph{$p$-Wasserstein space over $X$} -- that is, the set of Borel probability measures on $X$ with finite $p^{th}$ moments endowed with the $p$-Wasserstein metric.

An important special case for the present paper is when each measure is a sum of uniformly weighted Dirac measures. By an \emph{empirical measure of size $k$} on a space $X$, we mean a measure of the form $\alpha = \frac{1}{k} \sum_{i=1}^k \delta_{x_i}$, where $x_i \in X$ for all $i$. Birkhoff's theorem implies that for two empirical measures of size $k$, $\alpha = \frac{1}{k} \sum_{i=1}^k \delta_{x_i}$ and $\beta = \frac{1}{k} \sum_{i=1}^k \delta_{y_i}$, the $1$-Wasserstein distance $W_1(\alpha, \beta)$ can be expressed as

$$
\frac{1}{k} \min_{\pi} \sum_{i=1}^k d(x_i, y_{\pi(i)})
$$
where $\pi$ ranges over all bijections $\langle k \rangle \to \langle k \rangle$. 
We will call a $\pi$ realizing this minimum an \emph{optimal matching}. The special case of $X = \mathbb{R}$ will be important later. In this case, an optimal matching is given by pairing the $j^\mathrm{th}$ largest value in $(x_i)_{1\leq i \leq k}$ with the $j^\mathrm{th}$ largest value in $(y_i)_{1\leq i \leq k}$ for all $j$.

For a metric space $X$, we denote the space of empirical measures of size $k$ with the $1$-Wasserstein distance by $\mathcal{E}^k X$. We can also think of $\mathcal{E}^k X$ as the space of unordered $k$-tuples of points in $X$, sometimes called the $k$-fold symmetric product of $X$. See \cite{harms2023geometry, needham2023geometric} for some geometric properties of this space and some applications where it naturally arises.

We now define the sliced Wasserstein distance for measures in $\mathbb{R}^n$. For any $\theta \in \mathbb{R}^n$, denote by $\theta^\ast$ the orthogonal projection onto $\theta$.

\begin{Definition}[\cite{rabin2012wasserstein}]
Given two Borel probability measures $\alpha$ and $\beta$, on $\mathbb{R}^n$ the \emph{sliced $p$-Wasserstein distance} is defined by
\[
SW_p(\alpha, \beta) = \left( \int_{\mathbb{S}^{n-1}} W^p_p(\theta^\ast_{\#} \alpha,\ \theta^\ast_{\#} \beta)  \right)^{\frac{1}{p}}
\]
where $\theta^\ast_{\#}$ denotes the pushforward map induced by $\theta^\ast$.
\end{Definition}

We will also need the definition of a coarse embedding, originally introduced under the name uniform embedding by Gromov~\cite{gromov20041asymptotic}. For a broad introduction to coarse geometry we direct the reader to \cite{roe2003lectures}.

\begin{Definition}\label{def:coarseembedding}
    A map $f: X \to Y$ from a metric space $X$ to a metric space $Y$ is a \emph{coarse embedding} if there are non-decreasing functions $\rho_{-},\rho_{+}: [0, \infty) \to [0, \infty)$, called \emph{control functions}, with $\lim_{t\to\infty} \rho_{-} = \infty$ and 
\begin{align}\label{eq:control}
    \forall_{x,x' \in X} \ \rho_{-}(d(x,x')) \leq d(f(x), f(x')) \leq \rho_{+}(d(x,x')).
\end{align}
\end{Definition}

In particular, a bi-Lipschitz embedding is a coarse embedding where the control functions in (\ref{eq:control}) are linear. Coarse embedding is a much weaker condition that bi-Lipschitz embedding, however, even for geodesic metric spaces.

\section{Bounds on the sliced Wasserstein distance}\label{sec:bounds}

Our goal in this section is to derive bounds on the sliced Wasserstein distance between empirical measures in $\mathbb{R}^n$. Some bounds for general measures were obtained in Chapter 5 of \cite{bonnotte2013unidimensional}; our upper bound is a more explicit version of Proposition 5.1.3 there, but for the lower bound we require a more specialized result for empirical measures. We begin with the lower bound. 

Let $B(a,b) = \int_0^1 t^{a-1} (1-t)^{b-1} dt$ denote the Beta function,  and let $$
I(x;a,b) = \frac{1}{B(a,b)} \int_0^x t^{a-1} (1-t)^{b-1} dt
$$
denote the regularized incomplete Beta function. Let $S_{n-1}$ denote the volume of the unit $(n-1)$-sphere $\mathbb{S}^{n-1} \subseteq \mathbb{R}^{n}$.  For convenience, and since we will only require results for sufficiently large $n$, we will assume throughout that $n \geq 3$ (but see Remark \ref{rem:R2}). 

\begin{Lemma}\label{lemma:hypersph}
Let $\mathbb{S}^{n-1}\subseteq \mathbb{R}^{n}$ be the unit $(n-1)$-sphere and let $\mathbf{x}$ be a random variable uniformly distributed on $\mathbb{S}^{n-1}$. Then the absolute value of the first coordinate $|\mathbf{x}_1|$ has probability density function
$$
f(|\mathbf{x}_1| = t) = \frac{2(1-t^2)^{\frac{n-1}{2}-1}}{B(\frac{n-1}{2}, \frac{1}{2})},
$$
cumulative density function
$$
f(|\mathbf{x}_1| \leq t) = 1- I \left( 1-t^2; \frac{n-1}{2}, \frac{1}{2} \right),
$$
and expectation
$$
\mathbb{E}[|\mathbf{x}_1|] = \frac{2}{(n-1) B(\frac{n-1}{2}, \frac{1}{2})}.
$$
\end{Lemma}
\begin{proof}
For $t \in [0,1]$, let $A_t$ be the spherical segment
$$A_t = \{ \mathbf{v} \in \mathbb{S}^{n-1} \subseteq \mathbb{R}^{n} \mid |\mathbf{v}_1| \leq t\}.$$
From the formula for the area of a hyperspherical cap in \cite{li2010concise}, we have that the area of $A_t$ is
$$
S_{n-1}-S_{n-1}\cdot I\left(1-t^2; \frac{n-1}{2}, \frac{1}{2} \right)
$$
which gives the cumulative density function above. To obtain the probability density function, we differentiate the expression above.  Finally, the expectation is given by
\begin{align*}
\int_0^1 t \cdot  \frac{2(1-t^2)^{\frac{n-1}{2}-1}}{B(\frac{n-1}{2}, \frac{1}{2})}  dt = \frac{2}{(n-1) B(\frac{n-1}{2}, \frac{1}{2})}
\end{align*}
\end{proof}

\begin{Lemma}\label{lemma:betaineq}
For all $x \in [0,1]$ and integers $n \geq 3$,
\begin{equation}
1 - I\left(1-x^2; \frac{n-1}{2}, \frac{1}{2}\right) \leq 2e \cdot (n-1)x\cdot{B\left(\frac{n-1}{2}, \frac{1}{2}\right)}
\end{equation}
\end{Lemma}
\begin{proof}
The inequality 
$$
1 - I\left(1-x^2; \frac{n-1}{2}, \frac{1}{2}\right) -2e\cdot (n-1)x\cdot{B\left(\frac{n-1}{2}, \frac{1}{2}\right)} \leq 0
$$
clearly holds for $x = 0$. Differentiating the left side, we obtain
\begin{align*}
& \frac{2(1-x^2)^{\frac{n-1}{2}-1}}{B(\frac{n-1}{2}, \frac{1}{2})} - 2e\cdot(n-1)B\left(\frac{n-1}{2}, \frac{1}{2} \right)
\end{align*}
which is a non-increasing function of $x$ since we assumed $n \geq 3$.  Evaluating at $x=0$ we obtain
\begin{align*}
& \frac{2}{B(\frac{n-1}{2}, \frac{1}{2})} - 2e\cdot(n-1)B\left(\frac{n-1}{2}, \frac{1}{2} \right)  \\
 = \ &  \frac{1}{B(\frac{n-1}{2}, \frac{1}{2})}\left( 2 - 2e\cdot(n-1)B\left(\frac{n-1}{2}, \frac{1}{2} \right)^2 \right)
\end{align*}
To show this is always negative, it remains to show that 
$$
B\left(\frac{n-1}{2}, \frac{1}{2} \right)^2 \geq \frac{1}{e(n-1)}
$$
From \cite{grenie2015inequalities} we have
$$
B\left(\frac{n-1}{2}, \frac{1}{2} \right) \geq \frac{(\frac{n-1}{2})^{\frac{n-3}{2}}(\frac{1}{2})^{\frac{1}{2}}}{(\frac{n}{2})^{\frac{n-2}{2}}}
$$
so
$$
B\left(\frac{n-1}{2}, \frac{1}{2} \right)^2 \geq \frac{1}{n-1} \cdot \left(\frac{n-1}{n} \right)^{n-2} \geq \frac{1}{n-1}\cdot \frac{1}{e}
$$
where we use the fact that $\left(\frac{n-1}{n} \right)^{n-2}$ is a decreasing function of $n$ for $n\geq 2$, with limit $e^{-1}$ as $n\to\infty$.
\end{proof}

\begin{Lemma}\label{lemma:xcoordlower}
Let $\Omega \subseteq \mathbb{S}^{n-1}$ be subset of the unit sphere with area $|\Omega| \geq cS_{n-1}$ where $c \in (0,1]$. Then
$$
\int_\Omega |\mathbf{x}_1| d\mathbf{x} \geq \frac{c^2 \cdot S_{n-1} }{8e \cdot (n-1)\cdot B(\frac{n-1}{2}, \frac{1}{2})}
$$ 
\end{Lemma}
\begin{proof}
Let $t = \frac{c}{4e(n-1)B(\frac{n-1}{2}, \frac{1}{2})}$ and let $A_t = \{ \mathbf{v} \in \mathbb{S}^{n-1} \subseteq \mathbb{R}^{n} \mid |\mathbf{v}_1| \leq t\}$.  By Lemmas \ref{lemma:hypersph} and \ref{lemma:betaineq}, we have
\begin{align*}
|A_t|  & = S_{n-1} \cdot \left( 1- I \left( 1-t^2; \frac{n-1}{2}, \frac{1}{2} \right) \right) \\
& \leq S_{n-1} \cdot  {2e(n-1)t}\cdot{B\left(\frac{n-1}{2}, \frac{1}{2}\right)} \\
& = S_{n-1}\cdot \frac{c}{2}
\end{align*}
It follows that the absolute value of the first coordinate of vectors in $\Omega$ is at least $t$ for at least half of the area. Thus
$$
\int_\Omega |\mathbf{x}_1| d\mathbf{x} \geq t \cdot cS_{n-1}/2 = \frac{c^2 \cdot S_{n-1} }{8e\cdot (n-1)\cdot B(\frac{n-1}{2}, \frac{1}{2})}
$$ 
\end{proof}

For convenience, we introduce the following quantity
$$
\kappa(n) := \frac{2S_{n-1} }{(n-1)\cdot B(\frac{n-1}{2}, \frac{1}{2})} = \int_{\mathbb{S}^{n-1}} |\mathbf{x}_1| d\mathbf{x}
$$
where the second equality follows from Lemma \ref{lemma:hypersph}.  Given any vector $\mathbf{v} \in \mathbb{R}^n$, we can map it by a linear isometry to a multiple of the first coordinate unit vector. This gives us the following corollary to Lemma \ref{lemma:xcoordlower}.

\begin{Lemma}\label{lemma:dotlowerbound}
Let $\mathbf{v} \in \mathbb{R}^n$.  Let $\Omega \subseteq \mathbb{S}^{n-1}$ be a subset of the unit sphere with area $|\Omega| \geq cS_{n-1}$ where $c \in (0,1]$.  Then
$$
\int_\Omega | \langle \mathbf{x}, \mathbf{v} \rangle | d\mathbf{x} \geq \frac{c^2||\mathbf{v}||}{16e} \kappa(n)
$$ 
\end{Lemma} 

We are now ready to provide a lower bound on the sliced Wasserstein distance between empirical measures. The following result can be thought of as a higher-dimensional generalization of Theorem 3.3 in \cite{carriere2017sliced} but for measures rather than persistence diagrams.

\begin{Proposition}\label{prop:lowerbound}
Let $\alpha = \frac{1}{k} \sum_{i=1}^k \delta_{\mathbf{x}^i}$ and $\beta = \frac{1}{k}\sum_{i=1}^k \delta_{\mathbf{y}^i}$ be empirical measures of size $k$ in $\mathbb{R}^n$. Then we have the following inequality:
\begin{equation}
\frac{1}{16e k!}\kappa(n)\cdot W_1(\alpha,  \beta) \leq SW_1(\alpha, \beta)
\end{equation}
\end{Proposition}
\begin{proof}
For every unit vector $\theta$, let $\pi^\theta: \langle k \rangle \to \langle k \rangle$ be an optimal matching between $\theta^\ast_{\#} \alpha$ and $\theta^\ast_{\#} \beta$ for the $1$-Wasserstein distance on $\mathbb{R}$. Note that since these measures are on the real line, $\pi^\theta$ depends only on the ordering of the $2k$ quantities $\langle \mathbf{x}^i, \theta \rangle$ and $\langle \mathbf{y}^i, \theta \rangle$ (if two quantities are equal, we can break the tie using their indices) and is constant for regions where this ordering does not change. Thus the fibres of the map $\theta \mapsto \pi^\theta$ are unions of intersections of half-spaces intersected with $\mathbb{S}^{n-1}$, and the sum of their areas is $S_{n-1}$. Denote the fibre of $\pi$ by $\Omega_{\pi}$. Using Lemma \ref{lemma:dotlowerbound} and summing over all bijections $\pi: \langle k \rangle \to \langle k \rangle$, we have
\begin{align*}
SW_1(\alpha, \beta) &  = \sum_{\pi} \int_{\Omega_{\pi}} \frac{1}{k} \sum_{i=1}^k  |\langle \mathbf{x}^i-\mathbf{y}^{\pi(i)}, \theta \rangle| d \theta \\
& \geq \frac{1}{k} \sum_{i=1}^k \sum_{\pi} \frac{|\Omega_{\pi}|^2}{16e S_{n-1}^2}\kappa(n) \cdot  ||\mathbf{x}^i-\mathbf{y}^{\pi(i)}||  \\
& \geq \sum_{\pi} \frac{|\Omega_{\pi}|^2}{16e S_{n-1}^2}\kappa(n)\cdot W_1(\alpha,  \beta) \\
& \geq \frac{1}{16e k!}\kappa(n)\cdot W_1(\alpha,  \beta)
\end{align*}
where the final inequality follows from the Cauchy-Schwarz Inequality and the fact that $\sum_\pi |\Omega_{\pi}| = S_{n-1}$.

\end{proof}

We now turn our attention to deriving an upper bound for the sliced Wasserstein distance, following Proposition 5.1.3 in \cite{bonnotte2013unidimensional}.

\begin{Proposition}\label{prop:upperbound}
Let $\alpha = \frac{1}{k} \sum_{i=1}^k \delta_{\mathbf{x}^i}$ and $\beta = \frac{1}{k}\sum_{i=1}^k \delta_{\mathbf{y}^i}$ be empirical measures of size $k$ in $\mathbb{R}^n$. Then we have the following inequality:
\begin{equation}
SW_1(\alpha, \beta) \leq \kappa(n) \cdot W_1(\alpha, \beta)
\end{equation}
\end{Proposition}
\begin{proof}
Without loss of generality we may assume that the optimal matching from $\alpha$ to $\beta$ for the $1$-Wasserstein distance in $\mathbb{R}^n$ is the identity map $\langle k \rangle \to \langle k \rangle$. For each $i$,  let $\mathbf{v}^i = \mathbf{x}^i - \mathbf{y}^i$. For any $i$, there is a linear isometry taking $\mathbf{v}^i$ to a multiple of the first coordinate unit vector, so 
$$
\int_{\mathbb{S}^{n-1}} |\langle \mathbf{v}^i, \mathbf{u} \rangle| d\mathbf{u} = ||\mathbf{v}^i|| \cdot \int_{\mathbb{S}^{n-1}} |\mathbf{u}_1| d\mathbf{u} = ||\mathbf{v}^i|| \cdot \kappa(n)
$$
Thus we have
\begin{align*}
SW_1(\alpha, \beta) & \leq \frac{1}{k} \sum_i \int_{\mathbb{S}^{n-1}} |\langle \mathbf{v}^i, \mathbf{u} \rangle| d\mathbf{u} \\
& = \frac{1}{k} \cdot \kappa(n) \sum_i ||\mathbf{v}^i || \\ 
& = \kappa(n) \cdot W_1(\alpha, \beta)
\end{align*}
\end{proof}

Together, Proposition \ref{prop:lowerbound} and \ref{prop:upperbound} give the following theorem.

\begin{Theorem}\label{thm:bounds}
Let $\alpha = \frac{1}{k} \sum_{i=1}^k \delta_{\mathbf{x}^i}$ and $\beta = \frac{1}{k}\sum_{i=1}^k \delta_{\mathbf{y}^i}$ be empirical measures of size $k$ in $\mathbb{R}^n$, where $n \geq 3$. Then
\begin{equation*}
\frac{1}{16e k!}\kappa(n)\cdot W_1(\alpha,  \beta) \leq SW_1(\alpha, \beta) \leq \kappa(n) \cdot W_1(\alpha, \beta)
\end{equation*}
\end{Theorem}

\begin{Remark}\label{rem:R2}
While the results above are stated for $n \geq 3$ for convenience, the proof of Theorem 3.3 in \cite{carriere2017sliced} shows that for $n=2$ we also have the following:
$$
\frac{1}{2(k(k-1)+1)} W_1(\alpha, \beta) \leq SW_1(\alpha, \beta) \leq W_1(\alpha, \beta)
$$
\end{Remark}

\section{Coarse embeddings into Hilbert space}\label{sec:emp}

\subsection{Spaces of empirical distributions}

By adapting arguments made in \cite{carriere2017sliced} for persistence diagrams, we show that one can construct a kernel using the sliced Wasserstein distance. See also the argument for the sliced $2$-Wasserstein distance in \cite{kolouri2016sliced}.

Recall that a function $f: X^2 \to \mathbb{R}$ satisfying $f(x,y) = f(y,x)$ is \emph{negative semi-definite} if for all integers $n > 0$, points $x_1,\ldots,x_n \in X$ and real numbers $a_1,\ldots,a_n \in \mathbb{R}$ such that $\sum_i a_i = 0$, $\sum_{i,j} a_i a_j f(x_i, x_j) \leq 0$. Recall from \cite{carriere2017sliced} that the $W_1$ distance is a negative semi-definite function on the space of nonnegative measures on $\mathbb{R}$ with fixed mass $r > 0$. From linearity of integration, it follows that the sliced Wasserstein distance $SW_1$ is also a negative semi-definite function, as is any scalar multiple of $SW_1$. We will need the following result.

\begin{Theorem}[Schoenberg \cite{schoenberg1938metric}]\label{thm:schoenberg}
    Let $f: X^2 \to \mathbb{R}$ be a negative semi-definite function satisfying $f(x,x) = 0$ for all $x \in X$. Then there is a Hilbert space $H$ and a map $\phi: X \to H$ such that for all $x,y \in X$,
    $$
    ||\phi(x) - \phi(y)||^2 = f(x,y)
    $$
\end{Theorem}

Using the above result, we obtain an embedding result for the $1$-Wasserstein distance on empirical distributions in $\mathbb{R}^n$. Notice that the bounds in the theorem below do not depend on the dimension $n$. 

\begin{Theorem}\label{thm:rn}
    Let $X = \mathcal{E}^k \mathbb{R}^n$ be the space of all empirical distributions of size $k$ in $\mathbb{R}^n$ endowed with the $1$-Wasserstein distance. Then there is a Hilbert space $H$ and a coarse embedding $\phi: X \to H$ satisfying, for any $\alpha, \beta \in X$,
    $$
    \sqrt{\frac{1}{16e k!} \cdot W_1(\alpha, \beta)} \leq d(\phi(\alpha), \phi(\beta)) \leq \sqrt{W_1(\alpha, \beta)}
    $$
\end{Theorem}
\begin{proof}
    Define $H$ and $\phi: X \to H$ using Theorem \ref{thm:schoenberg} so that
    $$
    d(\phi(\alpha), \phi(\beta)) = \sqrt{\frac{1}{\kappa(n)} \cdot SW_1(\alpha, \beta)}
    $$
    where $\kappa(n)$ is as defined in Section \ref{sec:bounds}. Then by Theorem \ref{thm:bounds}, we have the required result.
\end{proof}

Recall the following result of Nowak which characterizes embeddability into Hilbert space in terms of embeddings of finite subsets.

\begin{Theorem}[Nowak~\cite{nowak2005coarse}]\label{thm:nowak}
A metric space $X$ admits a coarse embedding into a Hilbert space if and only if there exists non-decreasing functions $\rho_{-},\rho_{+}\colon [0,\infty) \to [0,\infty)$ such that $\lim_{t\to \infty} \rho_{-}(t) = \infty$ and for every finite subset $A \subset X$ there exists a map $f_A\colon A\to \ell_2$ satisfying, 
\begin{align*}
    \rho_{-}(d(x,y))\leq \|f_A(x)-f_A(y)\| \leq \rho_{+}(d(x,y))
\end{align*}
for every $x,y \in A$.
\end{Theorem}

Using this result, we obtain the following result regarding the space of empirical distributions on Hilbert space.

\begin{Theorem}\label{thm:hilbert}
    Let $H$ be a Hilbert space and let $X = \mathcal{E}^k H$ be the space of all empirical distributions of size $k$ in $H$ endowed with the $1$-Wasserstein distance. Then there is a Hilbert space $H'$ and a coarse embedding $\phi: X \to H'$.
\end{Theorem}
\begin{proof}
    Let $A$ be a finite subset of $X$. The union of the supports of all the measures in $A$ contains finitely many vectors, and therefore $A$ is a subset of $\mathcal{E}^k Y$ where $Y$ is a finite dimensional subspace of $H$. Since $Y$ is isometric to $\mathbb{R}^{\dim Y}$, $\mathcal{E}^k Y$ embeds into Hilbert space by Theorem \ref{thm:rn} with control functions that do not depend on $\dim Y$. Thus by Theorem \ref{thm:nowak} we obtain the result.
\end{proof}

Note that the above result cannot be derived from the previous results mentioned in the introduction since infinite dimensional Hilbert spaces do not have finite asymptotic dimension or Property A. 

\begin{Remark}
    By the results in \cite{andoni2018snowflake}, the spaces $\mathcal{E}^k \mathbb{R}^3$ for $k\in \mathbb{N}$ are not uniformly coarsely embeddable into Hilbert space -- i.e.~they cannot all be embedded into Hilbert space with the same control functions. We thus cannot remedy the fact that the lower bound in Theorem \ref{thm:rn} depends on the size $k$ of the distributions.  
\end{Remark}

A function $f: X \to Y$ induces a natural map $\mathcal{E}^k f: \mathcal{E}^k X \to \mathcal{E}^k Y$ which is just the pushforward of discrete measures. If $f$ is a coarse embedding, so is $\mathcal{E}^k f$, as the following proposition shows.

\begin{Proposition}\label{prop:Eembed}
    Let $X$ and $Y$ be a metric spaces. If $f: X \to Y$ is a coarse embedding, then so is $\mathcal{E}^k f: \mathcal{E}^k X \to \mathcal{E}^k Y$.
\end{Proposition}
\begin{proof}
    Suppose $f: X \to Y$ is a coarse embedding with control functions $\rho_{-},\rho_{+}$, and let $\alpha = \frac{1}{k} \sum_{i=1}^k \delta_{\mathbf{x}^i}$ and $\beta = \frac{1}{k}\sum_{i=1}^k \delta_{\mathbf{y}^i}$ be empirical measures of size $k$ in $X$. For any bijection $\pi: \langle k \rangle \to \langle k \rangle$, we have
    $$
    \frac{1}{k}\sum_{i=1}^k \rho_{-}\left(d(x_i, y_{\pi(i)})\right) \leq \frac{1}{k}\sum_{i=1}^k d(f(x_i), f(y_{\pi(i)})) \leq \frac{1}{k}\sum_{i=1}^k \rho_{+}\left(d(x_i, y_{\pi(i)})\right)
    $$
    Using the fact that both control functions are increasing, we have
    $$
    \frac{1}{k} \rho_{-}\left(\frac{1}{k}\sum_{i=1}^k d(x_i, y_{\pi(i)})\right) \leq \frac{1}{k}\sum_{i=1}^k \rho_{-}\left(d(x_i, y_{\pi(i)})\right) 
    $$
    and
    $$
    \frac{1}{k}\sum_{i=1}^k \rho_{+}\left(d(x_i, y_{\pi(i)})\right) \leq \frac{1}{k} \cdot k \cdot \rho_{+}\left(\sum_{i=1}^k d(x_i, y_{\pi(i)})\right)
    $$
    Thus
    $$
    \frac{1}{k} {\rho}_{-}\left(\frac{1}{k}\sum_{i=1}^k d(x_i, y_{\pi(i)})\right) \leq \frac{1}{k}\sum_{i=1}^k d(f(x_i), f(y_{\pi(i)})) \leq \frac{1}{k} \cdot k \cdot \rho_{+}\left(\sum_{i=1}^k d(x_i, y_{\pi(i)})\right)
    $$
    and taking minima over $\pi$ for each expression gives
    $$
    \frac{1}{k} {\rho}_{-}\left(W_1(\alpha, \beta)\right) \leq W_1(\mathcal{E}^kf(\alpha), \mathcal{E}^kf(\beta)) \leq \rho_{+}\left(k \cdot W_1(\alpha, \beta) \right)
    $$
    which completes the proof.
\end{proof}

We thus obtain the following version of Theorem \ref{thm:main1} from the introduction.

\begin{Theorem}\label{thm:empirical}
    Let $X$ be a metric space and let $k > 0$ be an integer. If $X$ coarsely embeds into Hilbert space, then so does $\mathcal{E}^k X$, the set of all empirical distributions on $X$ of size $k$.
\end{Theorem}
\begin{proof}
Since $X$ coarsely embeds into some Hilbert space $H$, $\mathcal{E}^k X$ coarsely embeds into $\mathcal{E}^k H$ by Proposition \ref{prop:Eembed}, which in turn embeds into Hilbert space by Theorem \ref{thm:hilbert}.
\end{proof}

\subsection{Orbit spaces}

We can now present a proof of the main result, Theorem \ref{thm:main2} from the introduction. Recall that if $X$ is a metric space and $G = \{g_1, \ldots, g_k\}$ is a finite group acting on $X$ by isometries, then the induced metric on the orbit space $X/G$ is given by
$$
d([x], [y]) = \inf_{x' \in [x],\ y' \in [y]} d(x', y') = \inf_{g \in G} d(x, g\cdot y)
$$

\begin{proof}[Proof of Theorem \ref{thm:main2}]
Let $x,y \in X$, let $\{x_1, \ldots, x_k\}$ and $\{y_1, \ldots, y_k\}$ be the orbits of $x$ and $y$ respectively, and suppose that $g' \in G$ satisfies $d([x], [y]) = d(x, g' \cdot y)$. Because the action of $G$ is by isometries, we have:
\begin{equation*}
d([x], [y]) = d(x, g'\cdot y) = \frac{1}{k} \sum_{i=1}^k d(g_i \cdot x, g_i g' \cdot y) \geq W_1(\mathcal{L}([x]), \mathcal{L}([y]))    
\end{equation*}
where the distributions $\mathcal{L}([x])$ and $\mathcal{L}([x])$ are the empirical distributions
$$
\mathcal{L}([x]) = \frac{1}{k}\sum_i \delta_{x_i},\ \mathcal{L}([y]) = \frac{1}{k}\sum_i \delta_{y_i}.
$$
On the other hand, $W_1(\mathcal{L}([x]), \mathcal{L}([y])) \geq d([x], [y])$ since $d([x], [y])$ is the smallest distance from the orbit of $x$ to the orbit of $y$. Thus, the orbit space $X/G$ embeds isometrically into $\mathcal{E}^k X$, the space of empirical distributions on $X$ of size $k = |G|$ with the $1$-Wasserstein distance. By Theorem \ref{thm:empirical}, $\mathcal{E}^k X$ coarsely embeds into a Hilbert space.
\end{proof}

\section{Concluding remarks}

\subsection{Coarsely $n$-to-$1$ maps}

The results in~\cite{dydak2016preserving} on finite asymptotic dimension, Asymptotic Property C and Yu's Property A mentioned in the introduction are special cases of more general results regarding so-called coarsely $n$-to-$1$ maps~\cite{miyata2013dimension, dydak2016preserving}. The following analogous question for coarse embeddability remains open.

\begin{Question}
If $X$ admits a coarse embedding into Hilbert space and $f: X \to Y$ is a surjective coarsely $n$-to-$1$ map, does $Y$ admit a coarse embedding into Hilbert space?
\end{Question}

\subsection{Max filters and coorbit embeddings}

Because the proof of Nowak's result (Theorem~\ref{thm:nowak}) is non-constructive, Theorem \ref{thm:main2} does not produce an explicit embedding of $X/G$, even when $X$ is itself a finite dimensional Euclidean space. In the context of invariant machine learning, explicit embeddings are preferable, which has led various authors to construct bi-Lipschitz embeddings $\mathbb{R}^n/G \to \mathbb{R}^m$ where $G$ is a finite group acting by unitaries. We briefly discuss two here, which are closely related to the sliced Wasserstein distance: the max filter bank~\cite{cahill2024group} and the coorbit embedding~\cite{balan2023}. 

With $G$ acting as above, fix a ``window'' vector $\mathbf{w} \in \mathbb{R}^n$ and for a vector $\mathbf{x} \in \mathbb{R}^n$, let $\Phi_{\mathbf{w}}(\mathbf{x})$ be the vector in $\mathbb{R}^{|G|}$ whose coordinates are $(\langle \mathbf{w}, g \cdot \mathbf{x} \rangle)_{g \in G}$ in ascending order. Picking $p$ window vectors $\mathbf{w}_1,\ldots,\mathbf{w}_p$ gives rise to the \emph{coorbit embedding} $\mathbb{R}^n/G \to \mathbb{R}^{|G|p}$ given by
$$
\Phi(\mathbf{x}) = (\Phi_{\mathbf{w}_1}(\mathbf{x}), \ldots \Phi_{\mathbf{w}_p}(\mathbf{x})).
$$
In practice, one may want to restrict to only a subset of the coordinates of $\Phi(\mathbf{x})$; restricting to just the maximum value of each $\Phi_{\mathbf{w}}(\mathbf{x})$ gives a \emph{max filter bank}. There is a relation with the $1$-Wasserstein distance given by
$$
||\Phi_{\mathbf{w}}(\mathbf{x}) - \Phi_{\mathbf{w}}(\mathbf{y})||_1 = {|G|} \cdot W_1(\mathbf{w}^\ast_{\#}\alpha, \mathbf{w}^\ast_{\#}\beta) 
$$
where $\alpha$ and $\beta$ are the empirical distributions corresponding to the orbits of $\mathbf{x}$ and $\mathbf{y}$ respectively, and $\mathbf{w}^\ast$ is the projection onto $\mathbf{w}$. Thus, the sliced Wasserstein distance is proportional to an integral of coorbit distances as $\mathbf{w}$ varies over all unit vectors. The authors of \cite{cahill2024group} and \cite{balan2023} are able to prove a bi-Lipschitz property for certain finite sets of window vectors (possibly chosen randomly). We note that the bounds they obtain are either not explicit or depend on the dimension $n$ (unlike Theorem~\ref{thm:rn}), and so are not enough to obtain our main result.

\bibliographystyle{plain}
\bibliography{embedbib}

\end{document}